\documentclass[a4paper,leqno,11pt,twoside]{amsart}
\usepackage{amsmath,amsfonts,amssymb,amsthm,amscd}
\usepackage[utf8]{inputenc}
\usepackage[english]{babel}
\usepackage[colorlinks=true,citecolor=blue, urlcolor=blue, linkcolor=blue,pagebackref]{hyperref}
\usepackage[top=1.2in,bottom=1.2in,left=1.in,right=1.in]{geometry} 
\usepackage{graphicx}
\usepackage{paralist}
\usepackage{tabto}
\usepackage{standalone}
\usepackage{tikz}
\usetikzlibrary{matrix}
\usetikzlibrary{arrows}
\usepackage{xfrac}

\usepackage[all]{xy}

\usepackage{enumerate}
\usepackage{xcolor}
\usepackage{aliascnt}

\theoremstyle{plain}

\newtheorem{thmIntr}{Theorem}

\newaliascnt{lem}{thm}
\newtheorem{lem}[lem]{Lemma}
\aliascntresetthe{lem}

\newaliascnt{cor}{thm}

\aliascntresetthe{cor}

\newaliascnt{prop}{thm}
\newtheorem{prop}[prop]{Proposition}
\aliascntresetthe{prop}

\theoremstyle{definition}

\newaliascnt{rem}{thm}
\newtheorem{rem}[rem]{Remark}
\aliascntresetthe{rem}

\newaliascnt{defn}{thm}

\aliascntresetthe{defn}

\newaliascnt{ex}{thm}
\newtheorem{ex}[ex]{Example}
\aliascntresetthe{ex}

\numberwithin{equation}{section}

\def\bP{\ensuremath{\mathbb{P}}}
\def\bQ{\ensuremath{\mathbb{Q}}}

\def\cA{\ensuremath{\mathcal{A}}}

\def\cC{\ensuremath{\mathcal{C}}}
\def\cD{\ensuremath{\mathcal{D}}}

\def\cH{\ensuremath{\mathcal{H}}}
\def\cM{\ensuremath{\mathcal{M}}}
\def\cO{\ensuremath{\mathcal{O}}}
\def\cP{\ensuremath{\mathcal{P}}}
\def\cR{\ensuremath{\mathcal{R}}}
\def\cS{\ensuremath{\mathcal{S}}}
\def\cT{\ensuremath{\mathcal{T}}}

\newcommand{\tC}{{\widetilde C}}

\def\ocD{\ensuremath{\overline{\mathcal{D}}}}
\def\ocM{\ensuremath{\overline{\mathcal{M}}}}
\def\ocR{\ensuremath{\overline{\mathcal{R}}}}
\def\ocS{\ensuremath{\overline{\mathcal{S}}}}

\DeclareMathOperator{\Pic}{Pic}

\DeclareMathOperator{\st}{st}

\definecolor{applegreen}{rgb}{0.55, 0.71, 0.0}
 
\newcommand{\set}[1]{\left\{#1\right\}}
\newcommand\restr[2]{{
 \left.\kern-\nulldelimiterspace 
 #1 
 \vphantom{\big|} 
 \right|_{#2} 
 }}

\title{The divisors of Prym semicanonical pencils}
\author[C. Maestro P\'erez and A. Rojas]{Carlos Maestro P\'erez and Andr\'es Rojas}
\address{C.M.P.: BMS -- Humboldt-Universit\"{a}t zu Berlin, Institut f\"{u}r Mathematik, Rudower Chaussee 25, 10099 Berlin, Germany}
\email{maestroc@mathematik.hu-berlin.de}
\address{A.R.: BGSMath -- Departament de Matem\`atiques i Inform\`atica,
Universitat de Barcelona, Gran Via de les Corts Catalanes, 585, 08007 Barcelona, Spain}
\curraddr{Mathematisches Institut, Universität Bonn, Endenicher
Allee 60, 53115 Bonn, Germany}
\email{arojas@math.uni-bonn.de} 
\thanks{C.~M.~P. was supported by a German DFG/BMS Phase II scholarship granted through the Humboldt-Universit\"{a}t. A.~R.~was supported by the Spanish MINECO grants MDM-2014-0445, RYC-2015-19175 and PID2019-104047GB-I00.}

\begin{document}
\maketitle

\begin{abstract}{
In the moduli space $\cR_g$ of double \'etale covers of curves of a fixed genus $g$, the locus of covers of curves with a semicanonical pencil decomposes as the union of two divisors $\cT^e_g$ and $\cT^o_g$. Adapting arguments of Teixidor for the divisor of curves having a semicanonical pencil, we prove that both divisors are irreducible and compute their divisor classes in the Deligne-Mumford compactification $\ocR_g$.
}
\end{abstract}

\section{Introduction}

Let $\pi: \tC \to C$ be a double \'etale cover between smooth curves of genus $g=g(C)$ and $\widetilde g=g(\tC)=2g-1$, and denote by $(P,\Xi)$ its (principally polarized) Prym variety.

In his fundamental work \cite{mu}, Mumford classified the singularities of the theta divisor $\Xi$. More precisely, he considered a translation $P^+$ of the Prym variety to $\Pic^{2g-2}(\tC)=\Pic^{\widetilde g-1}(\tC)$, together with a canonical theta divisor $\Xi^+\subset P^+$:
\begin{gather*}
P^+=\{M\in\Pic^{2g-2}(\tC)\mid\mathrm{Nm}_\pi(M)=\omega_C,\;h^0(\tC,M)\text{ even}\}, \\
\Xi^+=\{M\in P^+\mid h^0(\tC,M)\geq 2\}
\end{gather*}
Then every singular point of $\Xi^+$ is \textit{stable} ($M\in \Xi^{+}$ with $h^0(\tC, M)\geq4$) or \textit{exceptional} ($M=\pi^*L\otimes A \in \Xi ^{+}$ such that $h^0(C,L)\geq 2$ and $h^0(\tC,A)>0$).

Assume $C$ has a semicanonical pencil, that is, an even theta-characteristic $L$ with $h^0(C,L)\geq2$ (in the literature, this is also frequently referred to as a \emph{vanishing theta-null}). If $h^0(\tC,\pi^*L)$ is furthermore even, then $M=\pi^*L\in\Xi^{+}$ is an example of exceptional singularity. In that case, $L$ is called an \emph{even semicanonical pencil} for the cover $\pi$, and the Prym variety $(P,\Xi)$ belongs to the divisor $\theta_{null}\subset\cA_{g-1}$ of principally polarized abelian varieties whose theta divisor contains a singular $2$-torsion point.

In the paper \cite{be_invent}, Beauville showed that the Andreotti-Mayer locus
\[
  \mathcal{N}_0=\set{(A,\Xi) \in \mathcal A_4 \mid \text{Sing}\,(\Xi) \text{ is non-empty} }
\]
in $\cA_4$ is the union of two irreducible divisors: the (closure of the) Jacobian locus $\mathcal{J}_4$ and $\theta_{null}$.
An essential tool for the proof is the extension of the Prym map $\cP_g:\cR_g\to\cA_{g-1}$ to a proper map $\widetilde\cP_g:\widetilde\cR_g\to\cA_{g-1}$, by considering admissible covers instead of only smooth covers.
In the case $g=5$, this guarantees that every 4-dimensional principally polarized abelian variety is a Prym variety (i.e.~the dominant map $\cP_5$ is replaced by the surjective map $\widetilde {\cP}_5$). 

Then, one of the key points in Beauville's work is an identification of the coverings whose (generalized) Prym variety is contained in $\theta_{null}$. Indeed, the results in \cite[Section~7]{be_invent} together with \cite[Theorem~4.10]{be_invent} show that
\[
\cT^e=\text{(closure in $\widetilde\cR_5$ of)}\set{ [\pi:\tC \longrightarrow C ]\in \cR_5 \mid \text{the cover $\pi$ has an even semicanonical pencil}}
\]
is irreducible and equals $\widetilde \cP_5^{-1}(\theta_{null})$. Therefore, the irreducibility of $\theta_{null}\subset\cA_4$ is obtained from the irreducibility of $\cT^e$; the proof of the latter starts by noticing that 
\[
 \cT = \set{[C]\in  \cM_5 \mid \text{$C$ has a semicanonical pencil}}
\]
is an irreducible divisor of $ \cM_5$.

Now let us consider double étale covers of curves with a semicanonical pencil, in arbitrary genus. For a fixed $g\geq3$, let $\cT_g\subset\cM_g$ denote the locus of (isomorphism classes of) curves with a semicanonical pencil (i.e.~with an even, effective theta-characteristic). 
Note that $\cT_g$ is the divisorial part of the locus of curves admitting a theta-characteristic of positive (projective) dimension (\cite{te2}). 

The general element of $\cT_g$ has a unique such theta-characteristic (which is a semicanonical pencil $L$ with $h^0(C,L)=2$), and the pullback of $\cT_g$ to $ \cR_g$ decomposes as a union $\mathcal T^e_g \cup \mathcal T^o_g$ according to the parity of $h^0(\tC,\pi^*L)$. In other words, the general element of $\cT^e_g$ (resp.~$\cT^o_g$) is a cover with an even semicanonical pencil (resp.~an odd semicanonical pencil).

In view of Beauville's work, it is natural to ask whether $\cT^e_g$ and $\cT^o_g$ are irreducible divisors, and to ask about the behaviour of the restricted Prym maps ${\widetilde\cP_g}|_{\cT_g^e}$ and ${\widetilde\cP_g}|_{\cT_g^o}$.

This paper exclusively deals with the first question, and studies the divisors $\cT^e_g$ and $\cT^o_g$ of even and odd semicanonical pencils. Aside from its independent interest, it provides tools for attacking the second question; a study of the restricted Prym maps ${\widetilde\cP_g}|_{\cT_g^e}$ and ${\widetilde\cP_g}|_{\cT_g^o}$ is carried out in the subsequent paper \cite{lnr}.

Coming back to the first question, the divisor $\cT_g\subset\cM_g$ was studied by Teixidor in \cite{te}. Using the theory of limit linear series on curves of compact type developed by Eisenbud and Harris in \cite{eh}, Teixidor proved the irreducibility of $\cT_g$ and computed the class of its closure in the Deligne-Mumford compactification $\ocM_g$. 

In our case, we will work in the Deligne-Mumford compactification $\ocR_g$ of $\cR_g$ (first considered in \cite[Section 6]{be_invent}, during the construction of the proper Prym map). Following closely Teixidor's approach, we obtain natural analogues of her results for the two divisors of Prym semicanonical pencils:

\vspace{2mm}

\begin{thmIntr}\label{thmA}
Let $[\cT^e_g],[\cT^o_g]\in\Pic(\ocR_g)_\bQ$ denote the classes of (the closures of) $\cT^e_g$, $\cT^o_g$ in the Deligne-Mumford compactification $\ocR_g$. Then, the following equalities hold:
\begin{align*}
 [\cT^e_g]&=a \lambda -b_0'\delta_0'-b_0'' \delta_0''-b_0^{ram}\delta_0^{ram} -\sum_{i=1}^{ \lfloor g/2\rfloor} (b_i\delta_i+b_{g-i}\delta_{g-i} +b_{i:g-i}\delta_{i:g-i}), \\
 [\cT^o_g]&=c \lambda -d_0'\delta_0'-d_0'' \delta_0''-d_0^{ram}\delta_0^{ram} -\sum_{i=1}^{ \lfloor g/2\rfloor} (d_i\delta_i+d_{g-i}\delta_{g-i} +d_{i:g-i}\delta_{i:g-i}),
 \end{align*}
where
\begin{align*}
&a=2^{g-3}(2^{g-1}+1), &&c=2^{2g-4},\\
&b_0'=2^{2g-7}, &&d_0'=2^{2g-7},\\
& b_0''=0, &&d_0''= 2^{2g-6}, \\
&b_0^{ram}=2^{g-5}(2^{g-1}+1), && d_0^{ram}=2^{g-5}(2^{g-1}-1),\\
& b_i=2^{g-3}(2^{g-i}-1)(2^{i-1}-1), && d_i=2^{g+i-4}(2^{g-i}-1),\\
& b_{g-i}=2^{g-3}(2^{g-i-1}-1)(2^{i}-1), &&d_{g-i}=2^{2g-i-4}(2^{i}-1), \\
&b_{i:g-i}=2^{g-3}(2^{g-1}-2^{i-1}-2^{g-i-1}+1),&& d_{i:g-i}=2^{g-3}(2^{g-1}-2^{g-i-1}-2^{i-1}).
\end{align*}
\end{thmIntr}

\vspace{1.5mm}

\begin{thmIntr}\label{thmB}
For every $g\neq4$ the divisors $\cT^e_g$ and $\cT^o_g$ are irreducible.
\end{thmIntr}

A crucial role in the proofs is played by the intersection of $\cT^e_g$ and $\cT^o_g$ with the boundary divisors in $\ocR_g$ of covers of reducible curves. This is the content of \autoref{boundary}. Then in \autoref{proofA} we prove \autoref{thmA} by intersecting $\cT^e_g$ and $\cT^o_g$ with appropriate test curves in $\ocR_g$.

The proof of \autoref{thmB} for $g\geq5$ is given in \autoref{irred}, and combines monodromy arguments with the intersection of $\cT^e_g$ and $\cT^o_g$ with the boundary divisor $\Delta_1\subset\ocR_g$. 
We point out that the irreducibility for $g=3$ can be immediately checked in terms of hyperelliptic curves (\autoref{g=3}), whereas the irreducibility of $\cT_4^e$ and $\cT_4^o$ is obtained in the paper \cite{lnr} as a consequence of the study of the restricted Prym maps ${\cP_4}|_{\cT^e_4}$ and ${\cP_4}|_{\cT^o_4}$.

\vspace{2mm}

\textbf{Acknowledgements.}
The authors developed parts of this work independently as part of their doctoral research, and they would like to thank their respective advisors, Gavril Farkas, Martí Lahoz and Joan Carles Naranjo for their help and guidance. Thanks are also due to Alessandro Verra for suggesting the computation of divisor classes as a tool for the study of the Prym map on these divisors, as well as to the anonymous referees for their detailed comments.

\vspace{1mm}

\section{Preliminaries}

\subsection{The moduli space \texorpdfstring{$\ocR_g$}{Rg}}

This part is a brief review of the Deligne-Mumford compactification $\ocR_g$ and its boundary divisors. We follow the presentation of \cite[Section~1]{fa-lu}; the reader is referred to it for further details.

Let $\cM_g$ be the moduli space of smooth curves of genus $g$, and let $\ocM_g$ be its Deligne-Mumford compactification by stable curves. Following the standard notations, we denote by $\Delta_i$ ($i=0,\ldots,\lfloor g/2\rfloor$) the irreducible divisors forming the boundary $\ocM_g\setminus\cM_g$.
The general point of $\Delta_0$ is an irreducible curve with a single node, whereas the general point of $\Delta_i$ (for $i\geq1$) is the union of two smooth curves of genus $i$ and $g-i$, intersecting transversely at a point.

The classes $\delta_i$ of the divisors $\Delta_i$, together with the Hodge class $\lambda$, are well known to form a basis of the rational Picard group $\Pic(\ocM_g)_\bQ$.

We denote by $\cR_g$ the moduli space of connected double \'etale covers of smooth curves of genus $g$. 
In other words, $\cR_g$ parametrizes isomorphism classes of pairs $(C,\eta)$, where $C$ is smooth of genus $g$ and $\eta\in JC_2\setminus\set{\cO_C}$. 
It comes with a natural forgetful map $\pi:\cR_g\to\cM_g$ which is \'etale of degree $2^{2g}-1$.
Then, the Deligne-Mumford compactification $\ocR_g$ is obtained as the normalization of $\ocM_g$ in the function field of $\cR_g$. This gives a commutative diagram
\[
\xymatrix{
\cR_g\ar[rr]\ar[d]_{\pi}&& \ocR_g\ar[d]\\
\cM_g\ar[rr]&& \ocM_g
}
\]

where $\ocR_g$ is normal and the morphism $\ocR_g\to\ocM_g$ (that we will denote by $\pi$ as well) is finite.

Beauville's partial compactification $\widetilde{\cR}_g$ by admissible covers admits a natural inclusion into $\ocR_g$. As proved in \cite{balcasfon}, the variety $\ocR_g$ parametrizes isomorphism classes of \textit{Prym curves} of genus $g$, that is, isomorphism classes of triples $(X,\eta,\beta)$ where:
\begin{itemize}[\textbullet]
    \item $X$ is a quasi-stable curve of genus $g$, i.e.~$X$ is semistable and any two of its exceptional components are disjoint\footnote{Recall that a smooth rational component $E\subset X$ is called \textit{exceptional} if $\sharp E\cap\overline{X\setminus E}=2$, namely if it intersects the rest of the curve in exactly two points.}.
    \item $\eta\in\Pic^0(X)$ is a line bundle of total degree 0, such that $\restr{\eta}{E}=\cO_E(1)$ for every exceptional component $E\subset X$.
    \item $\beta:\eta^{\otimes 2}\to\cO_X$ is generically nonzero over each non-exceptional component of $X$.
\end{itemize}

In case that $\beta$ is clear from the context, by abuse of notation the Prym curve $(X,\eta,\beta)$ will be often denoted simply by $(X,\eta)$.

Then the morphism $\pi:\ocR_g\to\ocM_g$ sends (the class of) $(X,\eta,\beta)$ to (the class of) the \textit{stable model} $\st(X)$, obtained by contraction of the exceptional components of $X$.

Using pullbacks of the boundary divisors of $\ocM_g$, the boundary $\ocR_g\setminus\cR_g$ admits the following description  (see \cite[Examples 1.3 and 1.4]{fa-lu}):

\begin{enumerate}[(1)]
    \item 
    Let $(X,\eta,\beta)$ be a Prym curve, such that $\st(X)$ is the union of two smooth curves $C_i$ and $C_{g-i}$ (of respective genus $i$ and $g-i$) intersecting transversely at a point $P$. In such a case $X=\st(X)$, and giving a 2-torsion line bundle $\eta\in\Pic^0(X)_2$ is the same as giving a nontrivial pair $(\eta_i,\eta_{g-i})\in\left(JC_i\right)_2\times\left(JC_{g-i}\right)_2$.
    
    \noindent Then the preimage $\pi^{-1}(\Delta_i)$ decomposes as the union of three irreducible divisors (denoted by $\Delta_i$, $\Delta_{g-i}$ and $\Delta_{i:g-i}$), which are distinguished by the behaviour of the 2-torsion bundle.
    More concretely, their general point is a Prym curve $(X,\eta)$, where $X=C_i\cup_P C_{g-i}$ is a reducible curve as above and the pair $\eta=(\eta_i,\eta_{g-i})$ satisfies:
    \begin{itemize}[\textbullet]
        \item $\eta_{g-i}=\cO_{C_{g-i}}$, in the case of $\Delta_i$.
        \item $\eta_i=\cO_{C_i}$, in the case of $\Delta_{g-i}$.
        \item $\eta_i\neq\cO_{C_i}$ and $\eta_{g-i}\neq\cO_{C_{g-i}}$, in the case of $\Delta_{i:g-i}$.
    \end{itemize}
    
    \vspace{2mm}
    
    \item 
    Let $(X,\eta,\beta)$ be a Prym curve, such that $\st(X)$ is the irreducible nodal curve obtained by identification of two points $p,q$ on a smooth curve $C$ of genus $g-1$.
    
    \noindent If $X=\st(X)$ and $\nu:C\to X$ denotes the normalization, then $\eta\in\Pic^0(X)_2$ is determined by the choice of $\eta_C=\nu^*(\eta)\in JC_2$ and an identification of the fibers $\eta_C(p)$ and $\eta_C(q)$.
    \begin{itemize}[\textbullet]
        \item If $\eta_C=\cO_C$, there is only one possible identification of $\cO_C(p)$ and $\cO_C(q)$ (namely identification by $-1$) giving a nontrivial $\eta\in\Pic^0(X)_2$. The corresponding element $(X,\eta)$ may be regarded as a Wirtinger cover of $X$.
        \item If $\eta_C\neq\cO_C$, for each of the $2^{2g-2}-1$ choices of $\eta_C$ there are two possible identifications of $\cO_C(p)$ and $\cO_C(q)$. The $2(2^{2g-2}-1)$ corresponding Prym curves $(X,\eta)$ are non-admissible covers of $X$.
    \end{itemize}
    
    \noindent If $X\neq\st(X)$, then $X$ is the union of $C$ with an exceptional component $E$ through the points $p$ and $q$. The line bundle $\eta\in\Pic^0(X)$ must satisfy $\restr{\eta}E=\cO_E(1)$ and $\restr{\eta}C^{\otimes 2}=\cO_C(-p-q)$, which gives $2^{2g-2}$ possibilities. The corresponding Prym curves $(X,\eta)$ give Beauville admissible covers of $\st(X)$.
    
    \noindent It follows that $\pi^{-1}(\Delta_0)=\Delta_0'\cup\Delta_0''\cup\Delta_0^{ram}$, where $\Delta_0'$ (resp. $\Delta_0''$, resp. $\Delta_0^{ram}$) is an irreducible divisor whose general point is a non-admissible (resp. Wirtinger, resp. Beauville admissible) cover. Moreover, $\Delta_0^{ram}$ is the ramification locus of $\pi$ (see \cite[Page 763]{fa-lu} or \cite[Section 3]{balcasfon}).
\end{enumerate}

In terms of divisor classes, we have equalities
\[
\pi^*(\delta_i)=\delta_i+\delta_{g-i}+\delta_{i:g-i},\;\;\;\;\pi^*(\delta_0)=\delta_0'+\delta_0''+2\delta_0^{ram}
\]
where of course $\delta_i,\delta_{g-i},\delta_{i:g-i}$ ($1\leq i\leq\lfloor g/2\rfloor$) and $\delta_0',\delta_0'',\delta_0^{ram}$ are the classes of the boundary divisors of $\ocR_g$. These boundary classes, together with the pullback (also denoted by $\lambda$) of the Hodge class of $\ocM_g$, form a basis of the rational Picard group $\Pic(\ocR_g)_\bQ$.

\subsection{Divisors of Prym semicanonical pencils} If $C$ is a smooth curve of genus $g\geq3$, by \emph{semicanonical pencil} on $C$ we mean an even, effective theta-characteristic. By \emph{dimension} of a theta-characteristic $L$ we mean the (projective) dimension $h^0(C,L)-1$ of the linear system $|L|$. 

The locus of smooth curves admitting a semicanonical pencil is a divisor in $\cM_g$, whose irreducibility was proved in \cite[Theorem~2.4]{te}. In the same paper, the class of its closure $\cT_g$ in $\ocM_g$ was computed.

Since the parity of theta-characteristics remains constant in families (\cite{mu2}), the pullback of $\cT_g$ to $\ocR_g$ decomposes as $\pi^{-1}(\cT_g)=\cT_{g}^e\cup\cT_{g}^o$, where $\cT^e_g$ (resp. $\cT^o_g$) is the closure in $\ocR_g$ of the set
\[
\begin{aligned}
&\set{(C,\eta)\in\cR_g\mid C\text{ has a semicanonical pencil $L$ with $h^0(C,L\otimes\eta)$ even}}\\
(\text{resp. }&\set{(C,\eta)\in\cR_g\mid C\text{ has a semicanonical pencil $L$ with $h^0(C,L\otimes\eta)$ odd}})
\end{aligned}
\]

Note that both $\cT^e_g$ and $\cT^o_g$ have pure codimension 1 in $\ocR_g$, since their union is the pullback by a finite map of an irreducible divisor.
Furthermore, the restriction \[\restr{\pi}{\cT_{g}^e}:\cT_{g}^e\longrightarrow\cT_{g} \;\;\;\text{ (resp. } \restr{\pi}{\cT_{g}^o}:\cT_{g}^o\longrightarrow\cT_{g})
\]is surjective and generically finite of degree $2^{g-1}(2^g+1)-1$ (resp. of degree $2^{g-1}(2^g-1)$). This follows from the fact that a general element of $\cT_g$ has a unique semicanonical pencil (\cite[Theorem 2.16]{te2}), as well as from the number of even and odd theta-characteristics on a smooth curve.

\vspace{0.5mm}

\begin{ex}\label{g=3}
When $g=3$ a semicanonical pencil is the same as a $g^1_2$, and thus the divisor $\cT_3\subset\ocM_3$ equals the hyperelliptic locus $\cH_3$. Of course, the semicanonical pencil on every smooth curve $C\in\cT_3$ is unique. The 63 non-trivial elements of $JC_2$ can be represented by linear combinations of the Weierstrass points $R_1,\ldots,R_8$ as follows:
\begin{itemize}[\textbullet]
    \item Those represented as a difference of two Weierstrass points, $\eta=\cO_C(R_i-R_j)$, form a set of $\binom{8}{2}=28$ elements. Observe that in this case the theta-characteristic $g^1_2\otimes\eta=\cO_C(2R_j+R_i-R_j)=\cO_C(R_i+R_j)$ is odd.
    
    \item Those expressed as a linear combination of four distinct Weierstrass points, $\eta=\cO_C(R_i+R_j-R_k-R_l)$, form a set of $\frac{\binom{8}{4}}{2}=35$ elements\footnote{Division by 2 comes from the fact that any two complementary sets of four Weierstrass points induce the same two-torsion line bundle.}. According to the number of odd and even theta-characteristics on a genus 3 curve, in this case $g^1_2\otimes\eta$ is even.
\end{itemize}

Hence we obtain
\begin{gather*}
\cT_3^o=\text{(closure of)}\set{(C,\eta)\in\cR_3\mid C\text{ hyperelliptic, }\eta=\cO_C(R_i-R_j)} \subset\ocR_3\\
\cT_3^e=\text{(closure of)}\set{(C,\eta)\in\cR_3\mid C\text{ hyperelliptic, }\eta=\cO_C(R_i+R_j-R_k-R_l)}\subset\ocR_3
\end{gather*}
and, since monodromy on hyperelliptic curves acts transitively on tuples of Weierstrass points, it turns out that both divisors $\cT^o_3$ and $\cT^e_3$ are irreducible.
\end{ex}

\vspace{1mm}

\section{Proof of \texorpdfstring{\autoref{thmA}}{Theorem A}}
\label{proofA}

We denote by $[\cT_{g}^e],[\cT_{g}^o]\in\Pic(\ocR_g)_\mathbb{Q}$ the classes in $\ocR_g$ of the divisors $\cT_{g}^e$ and $\cT_{g}^o$. This section is entirely devoted to proving \autoref{thmA}.

First of all, observe that the pullback of the class $[\cT_{g}]\in\Pic(\ocM_g)_\mathbb{Q}$ (computed in \cite[Proposition~3.1]{te}) expresses $[\cT_{g}^e]+[\cT_{g}^o]$ as
\begin{equation*}
  \pi^*[\cT_g] = 2^{g-3} \left( (2^g+1) \lambda -2^{g-3}( \delta_0' +\delta_0''+2\delta_0^{ram})-\sum_{i=1}^{\lfloor g/2\rfloor} (2^{g-i}-1)(2^i-1)(\delta_i+\delta_{g-i}+\delta_{i:g-i}) \right).
\end{equation*}

This relation, together with the linear independence of the basic classes considered in $\ocR_g$, simplifies the computations: if we know a coefficient for one of the divisors, then we also know the coefficient corresponding to the same basic class for the other divisor.
Keeping this in mind, the coefficients of \autoref{thmA} can be determined by essentially following three steps:

\begin{enumerate}[(1)]
  \item\label{step1} The pushforward $\pi_*[\cT_{g}^e]$ easily gives the coefficient $a$ (hence $c$), as well as a relation between $b_0',b_0''$ and $b_0^{ram}$ (hence between $d_0',d_0''$ and $d_0^{ram}$).
  \item\label{step2} We adapt an argument of Teixidor \cite{te} to compute the coefficients $b_i,b_{g-i}$ and $b_{i:g-i}$ for every $i\geq1$: first we describe the intersection of $\cT_{g}^e$ with the boundary divisors $\Delta_i,\Delta_{g-i}$ and $\Delta_{i:g-i}$, and then we intersect $\cT_{g}^e$ with certain test curves.
  \item\label{step3} Finally, $d_0'$ and $d_0''$ are obtained intersecting $\cT_{g}^o$ with test curves contained inside $\Delta_0'$ and $\Delta_0''$ respectively. The relation obtained in \eqref{step1} determines $d_0^{ram}$ as well.
\end{enumerate}

\vspace{3.5mm}

For step~\eqref{step1}, note that on the one hand
\[
 \pi_*[\cT^e_g]=\deg (\cT^e_g \to \cT_g)\cdot[\cT_{g}]=(2^{g-1}(2^g+1)-1)2^{g-3}\left((2^g+1) \lambda -2^{g-3}\delta_0-\ldots \right)
\]
where $\ldots $ is a expression involving only the classes $\delta_1,\ldots,\delta_{\lfloor g/2\rfloor}$. On the other hand
\[ \pi_*[\cT^e_g]=a\pi_*\lambda-b_0'\pi_*\delta_0'-b_0'' \pi_*\delta_0''-b_0^{ram}\pi_*\delta_0^{ram} -\sum_{i=1}^{\lfloor g/2\rfloor} (b_i\pi_*\delta_i+b_{g-i}\pi_*\delta_{g-i} +b_{i:g-i}\pi_*\delta_{i:g-i})
\]
and, since $\pi_*\lambda=\pi_*(\pi^*\lambda)=\deg\pi\cdot\lambda$ and the divisors $\Delta_0',\Delta_0''$ and $\Delta_0^{ram}$ of $\ocR_g$ have respective degrees $2(2^{2g-2}-1),1$ and $2^{2g-2}$ over $\Delta_0\subset\ocM_g$, we obtain
\[
 \pi_*[\cT_{g}^e]=a(2^{2g}-1) \lambda -(2(2^{2g-2}-1)b_0'+b_0''+2^{2g-2}b_0^{ram})\delta_0+\ldots
\]
where $\ldots$ again denotes a linear combination of $\delta_1,\ldots,\delta_{\lfloor g/2\rfloor}$.

Using that $\lambda,\delta_0,\ldots\delta_{\lfloor g/2\rfloor}\in\Pic(\ocM_g)_\mathbb{Q}$ are linearly independent, we can compare the coefficients of $\lambda$ and $\delta_0$. Comparison for $\lambda$ yields
\[
 a=\frac{(2^{g-1}(2^g+1)-1)2^{g-3}(2^g+1)}{2^{2g}-1}=2^{g-3}(2^{g-1}+1),
\]
therefore $c=2^{2g-4}$ due to the relation $a+c=2^{g-3}(2^g+1)$. 

Comparison for $\delta_0$ gives
\[
 (2^{2g-1}-2)b_0'+b_0''+2^{2g-2}b_0^{ram} =2^{2g-6}(2^{g-1}(2^g+1)-1),
\]
or equivalently
\[
 (2^{2g-1}-2)d_0'+d_0''+2^{2g-2}d_0^{ram} =2^{3g-7}(2^g-1).
\]

\vspace{4mm}

In step~\eqref{step2}, the key point is the following description of the intersection of $\cT_{g}^e$ and $\cT_{g}^o$ with the preimages $\pi^{-1}(\Delta_i)$. It is nothing but an adaptation of \cite[Proposition~1.2]{te}:

\begin{prop}
\label{boundary}
For $i\geq1$, the general point of the intersection $\cT_{g}^e\cap\pi^{-1}(\Delta_i)$ (resp. $\cT_{g}^o\cap\pi^{-1}(\Delta_i)$) is a pair $(C,\eta)$ where:
\begin{enumerate}[{\rm(i)}]
  \item\label{bound:item1} The curve $C$ is the union at a point $P$ of two smooth curves $C_i$ and $C_{g-i}$ of respective genera $i$ and $g-i$, and satisfies one of these four conditions ($j=i,g-i$):
  
  \begin{enumerate}
    \item[$\alpha_j)$]  $C_j$ has a 1-dimensional (even) theta-characteristic $L_j$. In this case, the 1-dimensional limit theta-characteristics on $C$ are determined by the aspects $|L_j|+(g-j)P$ on $C_j$ and $|L_{g-j}+2P|+(j-2)P$ on $C_{g-j}$, where $L_{g-j}$ is any even theta-characteristic on $C_{g-j}$.
    \item[$\beta_j)$]  $P$ is in the support of an effective (0-dimensional) theta-characteristic $L_j$ on $C_j$. The aspects of the 1-dimensional limit theta-characteristics on $C$ are $|L_j+P|+(g-j-1)P$ on $C_j$ and $|L_{g-j}+2P|+(j-2)P$ on $C_{g-j}$, where $L_{g-j}$ is any odd theta-characteristic on $C_{g-j}$.
  \end{enumerate}
  
\item\label{bound:item2} $\eta=(\eta_i,\eta_{g-i})$ is a non-trivial 2-torsion line bundle on $C$, such that the numbers $h^0(C_i, L_i\otimes\eta_i)$ and $h^0(C_{g-i}, L_{g-i}\otimes\eta_{g-i})$ have the same (resp. opposite) parity.
\end{enumerate}
\end{prop}

\begin{proof}
First of all, note that item~\eqref{bound:item1} describes the general element of the intersection $\cT_{g}\cap\Delta_i$ in $\ocM_g$: this is exactly \cite[Proposition~1.2]{te}.

Moreover, if $(C,\eta)\in\cT_{g}^e\cap\pi^{-1}(\Delta_i)$ (resp. $(C,\eta)\in\cT_{g}^o\cap\pi^{-1}(\Delta_i)$), then there exists (a germ of) a 1-dimensional family $(\cC\to S,H,\mathcal{L})$ of Prym curves $(\cC_s,H_s)$ endowed with a 1-dimensional theta-characteristic $\mathcal{L}_s$, such that:

\begin{enumerate}
    \item For every $s\neq0$, $(\cC_s,H_s)$ is a smooth Prym curve such that $\mathcal{L}_s\otimes H_s$ is an even (resp. odd) theta-characteristic on $\cC_s$.
    \item The family $(\cC\to S,H)$ specializes to $(C,\eta)=(\cC_0,H_0)$.
\end{enumerate}

The possible aspects of the 1-dimensional limit series of $\mathcal{L}$ on $C=\mathcal{C}_0$ are described by item~\eqref{bound:item1}. Now the result follows from the fact that, on the one hand, the aspects of the limit series of $\mathcal{L}\otimes H$ on $C=\mathcal{C}_0$ are the same aspects as the limit of $\mathcal{L}$, but twisted by $\eta=H_0$; and on the other hand, the parity of a theta-characteristic on the reducible curve $C$ is the product of the parities of the theta-characteristics induced on $C_i$ and $C_{g-i}$. 
\end{proof}

\vspace{1.5mm}

\begin{rem}
For a fixed general element $C$ of the intersection $\cT_{g}\cap\Delta_i$ (i.e.~a curve $C$ satisfying the condition \eqref{bound:item1} above), the number of $\eta=(\eta_i,\eta_{g-i})$ such that $(C,\eta)\in\cT_{g}^e$ can be easily computed. Indeed, the number of $\eta$ giving parities (even,even) is the product of the number of even theta-characteristics on $C_i$ and the number of even theta-characteristics on $C_{g-i}$:
  \[
  2^{i-1}(2^i+1)2^{g-i-1}(2^{g-i}+1)=2^{g-2}(2^i+1)(2^{g-i}+1).
  \]

Similarly, the number of $\eta$ giving parities (odd,odd) is
  \[
  2^{i-1}(2^i-1)2^{g-i-1}(2^{g-i}-1)=2^{g-2}(2^i-1)(2^{g-i}-1).
  \]
From all these, we have to discard the trivial bundle $(\cO_{C_i},\cO_{C_{g-i}})$. Hence the total number of $\eta$ (both even and odd, counted together) is
\[
2^{g-2}(2^i+1)(2^{g-i}+1)+2^{g-2}(2^i-1)(2^{g-i}-1)-1=2^{g-1}(2^g+1)-1,
\]
which indeed coincides with the degree of $\cT_{g}^e$ over $\cT_{g}$.
Of course the configuration of the fiber $\restr{\pi}{\cT_{g}^e}^{-1}(C)$ along the divisors $\Delta_i$, $\Delta_{g-i}$ and $\Delta_{i:g-i}$ will depend on whether $C$ satisfies $\alpha_j)$ or $\beta_j)$.
\end{rem}

\vspace{1.5mm}

\begin{lem}
\label{oddtc}
If $C$ is a smooth curve of genus $g$ and $\eta\in JC_2$ is a non-trivial 2-torsion line bundle, then there are exactly $2^{g-1}(2^{g-1}-1)$ odd theta-characteristics $L$ on $C$ such that $L\otimes\eta$ is also odd.
\end{lem}
\begin{proof}
This can be checked, for example, by considering how the group $JC_2$ of $2$-torsion line bundles acts on the set $S_g(C)$ of theta characteristics. The associated difference map
\[
S_g(C)\times S_g(C)\longrightarrow JC_2,\quad(M,N)\longmapsto M\otimes N^{-1}
\]
can be restricted to the set of pairs of non-isomorphic odd theta-characteristics, that is,
\[
S^-_g(C)\times S^-_g(C)-\Delta\longrightarrow JC_2-\{\mathcal{O}_C\}.
\]
Since $\#S^-_g(C)=2^{g-1}(2^g-1)$ and $\#JC_2=2^{2g}$, the fibers of this restriction have order
\[
\#S^-_g(C)\cdot(\#S^-_g(C)-1)\cdot(\#JC_2-1)^{-1}
=2^{g-1}(2^{g-1}-1),
\]
which reflects the number of odd theta-characteristics $L$ such that $L\otimes\eta$ is also odd.
\end{proof}

\vspace{1.5mm}

Now, given an integer $i \geq1$, we proceed to compute the coefficients $b_i$, $b_{g-i}$ and $b_{i:g-i}$ of the class $[\cT_{g}^e]$.
We follow the argument in \cite[Proposition~3.1]{te}.

Fix two smooth curves $C_i$ and $C_{g-i}$ of respective genera $i$ and $g-i$ having no theta-characteristic of positive dimension, as a well as a point $p\in C_i$ lying in the support of no effective theta-characteristic.
We denote by $F$ the curve (isomorphic to $C_{g-i}$ itself) in $\Delta_i\subset\ocM_g$, obtained by identifying $p$ with a variable point $q\in C_{g-i}$.
This curve has the following intersection numbers with the basic divisor classes of $\ocM_g$:
\[
F\cdot\lambda=0,\;F\cdot\delta_j=0\text{ for }j\neq i,\;F\cdot\delta_i=-2(g-i-1)
\]
(for a justification of these intersection numbers, see \cite[page~81]{ha-mu}).

Since the curve $F\subset\ocM_g$ does not intersect the branch locus of the morphism $\pi$, it follows that the preimage $\pi^{-1}(F)$ has $2^{2g}-1$ connected components;
each of them is isomorphic to $F$, and corresponds to the choice of a pair $\eta=(\eta_i,\eta_{g-i})$ of 2-torsion line bundles on $C_i$ and $C_{g-i}$ being not simultaneously trivial.

Let $\widetilde{F_i}$ be one of the components of $\pi^{-1}(F)$ contained in the divisor $\Delta_i$ of $\ocR_g$;
it is attached to an element $\eta=(\eta_i,\cO_{C_{g-i}})$, for a fixed non-trivial $\eta_i\in (JC_i)_2$. 

On the one hand, clearly $\delta_i$ is the only basic divisor class of $\ocR_g$ that intersects $\widetilde{F_i}$.
  The projection formula then says that the number $\widetilde{F_i}\cdot\delta_i$ in $\ocR_g$ equals the intersection $F\cdot\delta_i=-2(g-i-1)$ in $\ocM_g$.
  Therefore,
  \[
  \widetilde{F_i}\cdot[\cT_{g}^e]=\widetilde{F_i}\cdot(a\lambda-b_0'\delta_0'-\ldots)=2(g-i-1)b_i.
  \]

On the other hand, according to \autoref{boundary} an element $(C,\eta)\in\widetilde{F_i}$ belongs to $\cT_{g}^e$ if and only if the two following conditions are satisfied:
  \begin{enumerate}[a)]
    \item\label{conditiona} The point $q\in C_{g-i}$ that is identified with $p$ lies in the support of an effective theta-characteristic $L_{g-i}$. That is, $C$ satisfies $\beta_{g-i})$.
    \item\label{conditionb} The odd theta-characteristic $L_i$ of $C_i$, when twisted by $\eta_i$, remains odd.
  \end{enumerate}
  This gives the intersection number
  \[
  \widetilde{F_i}\cdot[\cT_{g}^e]=(g-i-1)2^{g-i-1}(2^{g-i}-1)2^{i-1}(2^{i-1}-1),
  \]
where we use \autoref{oddtc} to count the possible theta-characteristics $L_i$. 

Comparing both expressions for $\widetilde{F_i}\cdot[\cT_{g}^e]$, it follows that $b_i=2^{g-3}(2^{g-i}-1)(2^{i-1}-1)$.

With a similar argument (considering a connected component of $\pi^{-1}(F)$ contained in $\Delta_{g-i}$ or $\Delta_{i:g-i}$), one can find the numbers
\[
b_{g-i}=2^{g-3}(2^{g-i-1}-1)(2^i-1),\;\;b_{i:g-i}=2^{g-3}(2^{g-1}-2^{i-1}-2^{g-i-1}+1).
\]

\vspace{1mm}

\begin{rem}\label{transvers}
The transversality of these intersections can be shown by looking at the scheme $X^e$ parametrizing pairs $((C,\eta),L)$, where $(C,\eta)$ is a Prym curve and $L$ is a semicanonical pencil on $C$ such that $L\otimes\eta$ is even. If we restrict the forgetful map $X^e\to\cT_{g}^e$ to the component $\widetilde{F_i}$, we obtain a scheme $\mathcal{X}\to{\cT_{g}^e}|_{\widetilde{F}_i}$ which is, by the above discussion, isomorphic to the scheme $\mathfrak{J}^1_{g-1}(\widetilde{F_i})$ of limit linear series of type $\mathfrak{g}^1_{g-1}$ on Prym curves $(C,\eta)\in\widetilde{F_i}$ satisfying conditions \eqref{conditiona} and \eqref{conditionb}. Following the description of this moduli space given in \cite[Theorem~3.3]{eh}, we see that the scheme $\mathfrak{J}^1_{g-1}(\widetilde{F_i})$ splits as the product of two reduced $0$-dimensional schemes, namely
\[
\{(L_{g-i},\,q)\textrm{ as in }\eqref{conditiona}\}\times\{L_i\textrm{ as in }\eqref{conditionb}\}.
\]
Therefore $\mathfrak{J}^1_{g-1}(\widetilde{F_i})\cong\mathcal{X}\to{\cT_{g}^e}|_{\widetilde{F}_i}$ is everywhere reduced and the intersection between $\widetilde{F_i}$ and $\cT_{g}^e$ is transverse. A breakdown of this argument may be found in \cite[Theorem~2.2]{fa}.
\end{rem}

\vspace{1mm}

Now we proceed with step~\eqref{step3}.
We will determine the constants $d_0',d_0'',d_0^{ram}$ of the class $[\cT^o_g]$ by using the test curve of \cite[Example~3.137]{ha-mo}.

Fix a general smooth curve $D$ of genus $g-1$, with a fixed general point $p\in D$. Identifying $p$ with a moving point $q\in D$, we get a curve $G$ (isomorphic to $D$) which lies in $\Delta_0\subset\overline{\mathcal M}_g$. As explained in \cite{ha-mo}, the following equalities hold:
\[
G\cdot\lambda=0, G\cdot\delta_0=2-2g, G\cdot\delta_1=1, G\cdot\delta_i=0\text{ for }i\geq2,
\]
where the intersection of $G$ and $\Delta_1$ occurs when $q$ approaches $p$; in that case the curve becomes reducible, having $D$ and a rational nodal curve as components. 

Combining this information with the known divisor class $[\cT_{g}]$ in $\ocM_g$, we have 
\[
G\cdot[\cT_{g}]=2^{g-3}((g-3)\cdot2^{g-2}+1).
\]

In order to compute $d_0''$, let $\widetilde{G}''$ be the connected component of $\pi^{-1}(G)$ obtained by attaching to every curve $C=D_{pq}$ the 2-torsion line bundle $e=(\cO_D)_{-1}$ (i.e.~$\cO_D$ glued by -1 at the points $p,q$). Indeed $e$ is well defined along the family $G$, so $\widetilde{G}''$ makes sense and is isomorphic to $G$. 

Then:

\begin{itemize}[\textbullet]
  \item By the projection formula, $\widetilde{G}''\cdot\lambda=0$.
  \item Again by projection, $\widetilde{G}''\cdot(\pi^*\delta_0)=2-2g$. Actually, since $\widetilde{G}''\subset \Delta_0''$ and $\widetilde{G}''$ intersects neither $\Delta_0'$ nor $\Delta_0^{ram}$, the following equalities hold:
  \[
  \widetilde{G}''\cdot\delta_0''=2-2g, \;\; \widetilde{G}''\cdot\delta_0'=0=\widetilde{G}''\cdot\delta_0^{ram}.
  \]
  \item We have $\widetilde{G}''\cdot(\pi^*\delta_1)=1$, with $\widetilde{G}''\cdot\delta_1=1$ and $\widetilde{G}''\cdot\delta_{g-1}=0=\widetilde{G}''\cdot\delta_{1:g-1}$. 
  
  \noindent Indeed, the intersection $G\cap\Delta_1$ occurs when $p=q$;
  for that curve, the 2-torsion that we consider is trivial on $D$ but not on the rational component.
  Hence the lift to $\widetilde{G}''$ of the intersection point $G\cap\Delta_1$ gives a point in $\widetilde{G}''\cap\Delta_1$.
  \item It is clear that $\widetilde{G}''\cdot\delta_i=\widetilde{G}''\cdot\delta_{g-i}=\widetilde{G}''\cdot\delta_{i:g-i}=0$ for $i\geq2$.
  \item Since twisting by $e$ changes the parity of any theta-characteristic in any curve of the family $G$ by \cite[Theorems~2.12 and 2.14]{ha}, it follows that all the intersection points of $G$ and $\cT_{g}$ lift to points of $\widetilde{G}''\cap\cT_{g}^o$.
  
\end{itemize}

All in all, we have
\[
2^{g-3}((g-3)\cdot2^{g-2}+1)=\widetilde{G}''\cdot [\cT^o_g]=(2g-2)d_0''-2^{g-3}(2^{g-1}-1)
\]
and solving the equation we obtain $d_0''=2^{2g-6}$.

For the computation of $d_0'$, we consider $\widetilde{G}'=\pi^{-1}(G)\cap\Delta_0'$ in $\ocR_g$.
Note that for an element $(C=D_{pq},\eta)\in\widetilde{G}'$, $\eta$ is obtained by gluing a nontrivial 2-torsion line bundle on $D$ at the points $p,q$. Then:

\begin{itemize}[\textbullet]
  \item $\widetilde{G}'\cdot\lambda=0$ by the projection formula.
  
  \item Again by projection, $\widetilde{G}'\cdot(\pi^*\delta_0)=\deg(\widetilde{G}' \to G)(G\cdot\delta_0)=2(2-2g)(2^{2g-2}-1)$.
  Moreover, since $\widetilde{G}'\subset \Delta_0'$ intersects neither $\Delta_0''$ nor $\Delta_0^{ram}$ it follows that
  \[
  \widetilde{G}'\cdot\delta_0'=2(2-2g)(2^{2g-2}-1), \;\;\widetilde{G}'\cdot\delta_0''=0=\widetilde{G}'\cdot\delta_0^{ram}.
  \]
  
  \item $\widetilde{G}'\cdot(\pi^*\delta_1)=\deg(\widetilde{G}' \to G)(G\cdot\delta_1)=2(2^{2g-2}-1)$.
  We claim that $\widetilde{G}'\cdot\delta_1=0$ and $\widetilde{G}'\cdot\delta_{g-1}=2^{2g-2}-1=\widetilde{G}'\cdot\delta_{1:g-1}$. 
  
  \noindent Indeed, $G\cap\Delta_1$ occurs when $p=q$;
  when such a point is lifted to $\widetilde{G}'$, the 2-torsion is nontrivial on $D$ (by construction).
  This gives $\widetilde{G}'\cdot\delta_1=0$. 
  
  \noindent Moreover, triviality on the rational nodal component will depend on which of the two possible gluings of the 2-torsion on $D$ we are taking; in any case, since $\widetilde{G}'=\pi^{-1}(G)\cap\Delta_0'$ considers simultaneously all possible gluings of all possible non-trivial 2-torsion line bundles on $D$, we have $\widetilde{G}'\cdot\delta_{g-1}=\widetilde{G}'\cdot\delta_{1:g-1}$. This proves the claim.
  
  \item  Of course, $\widetilde{G}'\cdot(\pi^*\delta_i)=\widetilde{G}'\cdot\delta_{g-i}=\widetilde{G}'\cdot\delta_{i:g-i}=0$ whenever $i\geq2$. 
  
  \item Finally, we use again that the parity of a theta-characteristic on a nodal curve of the family $G$ is changed when twisted by $e=(\cO_D)_{-1}$.
  Since the two possible gluings of a non-trivial 2-torsion bundle on $D$ precisely differ by $e$, the intersection numbers $\widetilde{G}'\cdot[\cT_{g}^e]$ and $\widetilde{G}'\cdot[\cT_{g}^o]$ have to coincide, and at the same time add up to the total
  \[
  \widetilde{G}'\cdot(\pi^*[\cT_{g}])=\deg(\widetilde{G}' \to G)(G\cdot[\cT_{g}])=2(2^{2g-2}-1)\cdot 2^{g-3}((g-3)\cdot2^{g-2}+1)
  \]
  by the projection formula. That is, 
  \[
  \widetilde{G}'\cdot[\cT_{g}^e]=\widetilde{G}'\cdot[\cT_{g}^o]=(2^{2g-2}-1)\cdot2^{g-3}((g-3)\cdot2^{g-2}+1).
  \]
\end{itemize}


Putting this together with the coefficients $d_{g-1}=2^{2g-5}$ and $d_{1:g-1}=2^{g-3}(2^{g-2}-1)$ obtained in step~\eqref{step2}, we get
\begin{align*}
(2^{2g-2}-1)\cdot2^{g-3}((g-3)\cdot2^{g-2}+1)&=\widetilde{G}'\cdot[\cT^o_g]=\\
=2(2g-2)(2^{2g-2}-1)d_0'&-2^{2g-5}(2^{2g-2}-1)-2^{g-3}(2^{g-2}-1)(2^{2g-2}-1)
\end{align*}
and therefore $d_0'=2^{2g-7}$.

Finally, to compute $d_0^{ram}$ we simply combine the relation
\[
(2^{2g-1}-2)d_0'+d_0''+2^{2g-2}d_0^{ram} =2^{g-1}(2^g-1)2^{2g-6}
\]
obtained in step~\eqref{step1} with the coefficients $d_0',d_0''$ just found, to obtain $d_0^{ram}=2^{g-5}(2^{g-1}-1)$. This concludes step~\eqref{step3} and hence the proof of \autoref{thmA}.


\begin{rem}
The divisor $\cT_g$ has a more natural interpretation in the compactification of the moduli space $\cS_g^{+}$ of even spin curves (i.e.~curves equipped with an even theta-characteristic). In the same way, it would be preferable to discuss the divisors $\cT_{g}^e$ and $\cT_{g}^o$ in a space of curves endowed with both a Prym and a spin structure. In particular, if a good compactification of $\cR_g\times_{\cM_g}\cS_g^{+}$ were constructed and studied, then the divisor classes of $\cT_{g}^e$ and $\cT_{g}^o$ could also be derived from the diagram
\[
\xymatrix{
\cR_g&\cR_g\times_{\cM_g}\cS_g^{+}\ar[l]\ar[r]&\cS_g^{+}
}
\]
and the fact that the class of (the closure in $\ocS_g^{+}$ of) the divisor
\[
 \text{}\set{(C,L)\in  \cS_g^{+} \mid \text{$L$ is a semicanonical pencil on $C$}}
\]
was computed by Farkas in \cite[Theorem~0.2]{fa}. Following the ideas of \cite{ser}, a candidate space for such a compactification is proposed in \cite[Section~2.4]{mphd}, although it remains to check that this space is indeed a smooth and proper Deligne-Mumford stack. Under the assumption that it is, a study of its boundary reveals the same expressions obtained in \autoref{thmA}. Further details can be found in \cite{mphd}.
\end{rem}

\vspace{2mm}

\section{Proof of \texorpdfstring{\autoref{thmB}}{Theorem B}}\label{irred}

In this section we study the irreducibility of the divisors $\cT_{g}^o$ and $\cT_{g}^e$. Recall that for $g=3$, we already saw in \autoref{g=3} that the divisors $\cT_{3}^o$ and $\cT_{3}^e$ are irreducible.
In the general case ($g\geq5)$, our arguments are essentially an adaptation of those of Teixidor in \cite[Section~2]{te}, used to prove the irreducibility of $\cT_{g}$ in $\ocM_g$.

The idea for proving the irreducibility of $\cT_{g}^o$ is the following (the proof for $\cT_{g}^e$ will be similar, but some simplifications will arise). By using \autoref{boundary}, first we will fix a Prym curve $(C,\eta)$ (degeneration of smooth hyperelliptic ones) lying in all the irreducible components of the intersection $\cT^o_g\cap\Delta_1$. This reduces the problem to the local irreducibility of $\cT^o_g$ in a neighborhood of $(C,\eta)$ (after checking that every irreducible component of $\cT^o_g$ intersects $\Delta_1$). For the proof of the local irreducibility of $\cT^o_g$, we can take advantage of the scheme of pairs $((C,\eta),L)$ introduced in \autoref{transvers} and use the following observation:

\vspace{1mm}

\begin{rem}\label{irredneigh}
In a neighborhood of a given point, the local irreducibility of $\cT_{g}^o$ (resp. $\cT_{g}^e$) is implied by the local irreducibility of the scheme $X^o$ (resp. $X^e$) parametrizing pairs $((C,\eta),L)$, where $(C,\eta)$ is a Prym curve and $L$ is a semicanonical pencil on $C$ such that $L\otimes\eta$ is odd (resp. even).
This follows from the surjectivity of the forgetful map $X^o\to\cT_{g}^o$ (resp.  $X^e\to\cT_{g}^e$).
\end{rem}

Then the local irreducibility of $X^o$ (near our fixed $(C,\eta)$) will be argued by showing that monodromy  interchanges the (limit) semicanonical pencils on $C$ that become odd when twisted by the $2$-torsion bundle $\eta$.
Let us recall, for later use in this monodromy argument, some features of theta-characteristics on hyperelliptic curves:

\vspace{1mm}

\begin{rem}\label{hyp}
Let $C$ be a smooth hyperelliptic curve of genus $g$, with Weierstrass points $R_1,\ldots,R_{2g+2}$. 

Then, it is well known (see e.g.~\cite[Proposition~6.1]{mutata}) that the theta-characteristics on $C$ have the form $r\cdot g^1_2+S$, $r$ being its dimension (with $-1\leq r\leq[\frac{g-1}{2}]$) and $S$ being the fixed part of the linear system (which consists of $g-1-2r$ distinct Weierstrass points).
    
Moreover, given a 2-torsion line bundle of the form $\eta=\cO_C(R_i-R_j)$, theta-characteristics changing their parity when twisted by $\eta$ are exactly those for which $R_i,R_j\in S$ (the dimension increases by 1) or $R_i,R_j\notin S$ (the dimension decreases by 1).
\end{rem}

For the proof of \autoref{thmB} we also need the following result, which will guarantee that every irreducible component of $\cT^o_g$ and $\cT^e_g$ intersects the boundary divisor $\Delta_1\subset\ocR_g$:

\begin{lem}\label{divboundary}
Let $\cD\subset\cR_g$ be any divisor, where $g\geq5$. Then the closure $\ocD\subset\ocR_g$ intersects $\Delta_1$ and $\Delta_{g-1}$.
\end{lem}
\begin{proof}
We borrow the construction from \cite[Section~4]{mnp}, where (a stronger version of) the corresponding result for divisors in $\cM_g$ is proved.

Fix a complete integral curve $B\subset\cM_{g-2}$ (whose existence is guaranteed by the assumption $g\geq5$), two elliptic curves $E_1,E_2$ and a certain 2-torsion element $\eta\in JE_1\setminus\{0\}$. If $\Gamma_b$ denotes the smooth curve of genus $g-2$ corresponding to $b\in B$, one defines a family of Prym curves parametrized by $\Gamma_b^2$ as follows.

If $(p_1,p_2)\in\Gamma_b^2$ is a pair of distinct points, glue to $\Gamma_b$ the curves $E_1$ and $E_2$ at the respective points $p_1$ and $p_2$ (this is independent of the chosen point on the elliptic curves). To this curve attach a 2-torsion bundle being trivial on $\Gamma_b$ and $E_2$, and restricting to $\eta$ on $E_1$.

To an element $(p,p)\in\Delta_{\Gamma_b^2}\subset\Gamma_b^2$, we attach the curve obtained by gluing a $\bP^1$ to $\Gamma_b$ at the point $p$, and then $E_1,E_2$ are glued to two other points in $\bP^1$. Of course, the 2-torsion bundle restricts to $\eta$ on $E_1$, and is trivial on the remaining components.

Moving $b$ in $B$, this construction gives a complete threefold $T=\underset{b\in B}{\bigcup}\Gamma_b^2$ contained in $\Delta_1\cap\Delta_{g-1}$. Let also $S=\underset{b\in B}{\bigcup}\Delta_{\Gamma_b^2}$ be the surface in $T$ given by the union of all the diagonals; it is the intersection of $T$ with $\Delta_2$. Then, the following statements hold:

\begin{enumerate}
    \item $\restr{\delta_1}{S}=0$ and $\restr{\delta_{g-1}}{S}=0$ (the proof of \cite[Lemma~4.2]{mnp} is easily translated to our setting).
    \item $\restr{\lambda}{\Delta_{\Gamma_b^2}}=0$ for every $b\in B$, since all the curves in $\Delta_{\Gamma_b^2}$ have the same Hodge structure.
    \item If $a\in\bQ$ is the coefficient of $\lambda$ for the class $[\ocD]\in\Pic(\ocR_g)_\bQ$, then $a\neq0$. Indeed, $2^{2g-1}a\in\bQ$ is the coefficient of $\lambda$ for the class $[\overline{\pi(\cD)}]\in\Pic(\ocM_g)_\bQ$; then \cite[Remark~4.1]{mnp} proves the claim.
\end{enumerate}

These are the key ingredients in the original proof of \cite[Proposition~4.5]{mnp}. The same arguments there work verbatim in our case and yield the analogous result: $\restr{[\ocD]}{T}\neq m\cdot S$ for every $m\in\bQ$.

In particular, the intersection $\ocD\cap T$ is non-empty (and not entirely contained in $S$).
\end{proof}

\vspace{0.5mm}

\begin{prop}\label{irredTog}
For $g\geq5$, the divisor $\cT_{g}^o$ is irreducible.
\end{prop}
\begin{proof}
According to \autoref{boundary}, the intersection $\cT_{g}^o\cap\Delta_1$ consists of two loci $\alpha$ and $\beta$. The general point of each of these loci is the union at a point $P$ of a Prym elliptic curve $(E,\eta)$ and a smooth curve $C_{g-1}$ (with trivial line bundle) of genus $g-1$, such that:
\begin{itemize}[\textbullet]
    \item In the case of $\alpha$, the curve $C_{g-1}$ has a 1-dimensional theta-characteristic, i.e, $C_{g-1}\in\cT_{g-1}$ in $\ocM_{g-1}$. Moreover, there is exactly one limit semicanonical pencil on $C_{g-1}\cup_PE$ changing parity when twisted by the 2-torsion bundle; it induces the theta-characteristic $\eta$ on $E$.
    
    \noindent It follows that $\alpha$ is irreducible (by irreducibility of $\cT_{g-1}$) and the intersection of $\cT^o_g$ and $\Delta_1$ along $\alpha$ is reduced. In particular, there is a unique irreducible component of $\cT^o_g$ (that we will denote by $\cT^o_{g,\alpha}$) intersecting $\Delta_1$ along the whole locus $\alpha$.
    
    \item In the case of $\beta$, $P$ is in the support of a 0-dimensional theta-characteristic on $C_{g-1}$.
    Again, there is a unique limit semicanonical pencil changing parity, with induced theta-characteristic $\cO_E$ on $E$.
\end{itemize}

Now we consider a reducible Prym curve $(C,\eta)\in\Delta_1$ constructed as follows: $C$ is the join of an elliptic curve $E$ and a general smooth hyperelliptic curve $C'$ of genus $g-1$ at a Weierstrass point $P\in C'$, whereas the line bundle $\eta$ is trivial on $C'$. Note that $(C,\eta)$ is the general point of the intersection $\widetilde{\cH}_g\cap\Delta_1$, where $\widetilde{\cH}_g\subset\cT_{g}^o$ denotes the locus of hyperelliptic Prym curves whose 2-torsion bundle is a difference of two Weierstrass points.

Of course $(C,\eta)$ belongs to $\alpha$ and $\beta$; 
we claim that it actually belongs to any component of $\beta$.

To prove this claim, consider any irreducible component of $\beta$, and fix a general element of it. This general element admits the description given above: let us denote by $X$ (written $C_{g-1}$ above) the irreducible component of genus $g-1$, and by $Q_X\in X$ the point connecting $X$ with the elliptic component. Recall that $Q_X$ lies in the support of a 0-dimensional theta-characteristic $L_X$ on $X$.

We deform the pair $(X,L_X)$ to a pair $(C',L)$ formed by our hyperelliptic curve $C'$ and a $0$-dimensional theta-characteristic $L$ on it. According to the description of \autoref{hyp}, under this deformation the point $Q_X\in X$ specializes to a Weierstrass point $Q\in C'$.

Therefore, our irreducible component of $\beta$ contains a Prym curve which is the union of $C'$ (with trivial 2-torsion) and a Prym elliptic curve $(E',\eta')$ at the Weierstrass point $Q\in C'$.
Since the monodromy on hyperelliptic curves acts transitively on the set of Weierstrass points, we may replace $Q$ by our original Weierstrass point $P$ without changing the component of $\beta$.
Using that $\ocR_1$ is connected we can also replace $(E',\eta')$ by $(E,\eta)$. This proves the claim.

Now, to prove the irreducibility of $\cT_{g}^o$ we argue as follows: since $\cT_{g}^o$ has pure codimension 1, we know by \autoref{divboundary} that each of its irreducible components intersects $\Delta_1$.
As our point $(C,\eta)$ belongs to all the irreducible components of $\cT_{g}^o\cap\Delta_1$, it suffices to check the local irreducibility of $\cT_{g}^o$ in a neighborhood of $(C,\eta)$. 

To achieve this, in view of \autoref{irredneigh} we will check the local irreducibility of the scheme $X^o$. In other words, we need to study the \emph{limit semicanonical pencils on $C$ changing parity when twisted by $\eta$}. We do this in the rest of the proof, by checking that monodromy on $\widetilde{\cH}_g\subset\cT_{g}^o$ connects any \emph{limit semicanonical pencil changing parity on $(C,\eta)$} of type $\beta$ with one of type $\alpha$, and checking that limits of type $\alpha$ are also permuted by monodromy on $\cT^o_{g,\alpha}$.

Let $R_1,R_2,R_3$ be the points on $E$ differing from $P$ by 2-torsion, and let $R_4,\ldots,R_{2g+2}$ be the Weierstrass points on $C'$ that are different from $P$: reordering if necessary, we assume 
$\restr{\eta}{E}=\cO_E(R_1-R_2)$. Note that $R_1,\ldots,R_{2g+2}$ are the limits on $C$ of Weierstrass points on nearby smooth hyperelliptic curves, since they are the ramification points of the limit $g^1_2$ on $C$.

With this notation, arguing as in the proof of \autoref{boundary}, the possible aspects on $E$ of a \emph{limit semicanonical pencil changing parity on $(C,\eta)$} are:
\begin{itemize}[\textbullet]
    \item Those of type $\alpha$ have aspect on $E$ differing from the even theta-characteristic $\eta$ by $(g-1)P$, hence $\cO_E(R_3+(g-2)P)=\cO_E(R_1+R_2+(g-3)P)$.
    \item Those of type $\beta$ have aspect differing from the odd theta-characteristic $\cO_E$ by $(g-1)P$, hence  $\cO_E((g-1)P)=\cO_E(R_1+R_2+R_3+(g-4)P)$.
\end{itemize}

Given a family of \textit{semicanonical pencils changing parity} on nearby smooth curves of $\widetilde{\cH}_g$, we can distinguish the type of its limit on $C$ by knowing how many of the $g-1-2r$ fixed Weierstrass points in the moving theta-characteristic (recall \autoref{hyp}) specialize to $E$. 
If this number is 0 or 3 (resp. 1 or 2), then our limit is of type $\beta$ (resp. of type $\alpha$).

Hence, after using monodromy on smooth hyperelliptic curves to interchange the (limit) Weierstrass point $R_3$ with an appropriate (limit) Weierstrass point on $C'$, we obtain that monodromy on $\widetilde{\cH}_g\subset\cT_{g}^o$ interchanges any \textit{limit semicanonical pencil changing parity} of type $\beta$ with one of type $\alpha$ (of the same dimension). The only possible exception is a limit of $\frac{g-1}{2}\cdot g^1_2$ when $g\equiv3\pmod{4}$, since in that case there are no fixed points to interchange with $R_3$.

In addition, monodromy on $\cT^o_{g,\alpha}$ (the unique irreducible component of $\cT^o_g$ containing $\alpha$) acts transitively on the set of \textit{limit semicanonical pencils changing parity} of type $\alpha$. Indeed, if $X^o_\alpha$ denotes the preimage of $\cT^o_{g,\alpha}$ in $X^o$, then the forgetful map $X^o_\alpha\to \cT^o_{g,\alpha}$ is birational (by \cite[Theorem 2.16]{te2}) and has finite fibers; consequently $X^o_\alpha$ is irreducible, which proves the assertion.

Therefore to conclude the proof of the local irreducibility of $X^o$ near $(C,\eta)$ it only remains to show that, if $g\equiv3\pmod{4}$, the monodromy on $\cT_{g}^o$ interchanges the limit of $\frac{g-1}{2}\cdot g^1_2$ with a limit of theta-characteristics of lower dimension.
This can be achieved exactly with the same family of limit theta-characteristics as in \cite[Proposition~2.4]{te} for certain reducible curves; let us include a few words about the geometry of this family.

First, one degenerates $C'$ to a reducible hyperelliptic curve obtained by identifying a point $P'\in E'$ ($E'$ elliptic curve) with a Weierstrass point $Q\in C''$ ($C''\in\cM_{g-2}$ hyperelliptic), such that the Weierstrass point $P\in C'$ specializes to a point of $E'$. This naturally induces a degeneration $C_{P'}$ of our Prym curve $(C,\eta)$, in which the 2-torsion bundle is non-trivial only along the component $E$. We will denote by $R_4,R_5$ the points of $E'$ differing by $2$-torsion from $P$ and $P'$ (limits of the corresponding Weierstrass points of $C'$).

Consider the family of Prym curves $C_X$ obtained by glueing $E$ (the only component with non-trivial 2-torsion) and $E'$ at $P$, and by identifying $Q\in C''$ with a variable point $X\in E'$. Note that for $X=P'$, we indeed recover our deformation $C_{P'}$ of $(C,\eta)$. Every such Prym curve $C_X$ can be equipped with a \textit{limit semicanonical pencil changing parity} of aspects $\cO_E((g-1)P)$ on $E$, $\cO_{E'}(Q+(g-2)X)$ on $E'$ and $\cO_{C''}((g-1)Q)$ on $C''$.

On $C_{P'}$, this corresponds to the limit of $\frac{g-1}{2}\cdot g^1_2$ on nearby smooth Prym curves of $\widetilde{\cH}_g$; on the other hand, $C_{R_5}$ is also hyperelliptic and we have a limit of theta-characteristics of the form $\frac{g-5}{2}\cdot g^1_2+R_1+R_2+R_3+R_4$.

Therefore, monodromy on $\cT^o_g$ moves the limit of $\frac{g-1}{2}\cdot g^1_2$ to a limit theta-characteristic of type $\beta$ of lower dimension, which concludes the proof.
\end{proof}

\vspace{1mm}

\begin{prop}\label{irredTeg}
For $g\geq5$, the divisor $\cT_{g}^e$ is irreducible.
\end{prop}
\begin{proof}
The proof is similar to that of $\cT^o_g$, but with some simplifications (due to the fact that the intersection $\cT_{g}^e\cap\Delta_1$ consists only of a locus $\alpha$). Let us give an outline of the argument.

In virtue of \autoref{boundary}, the general point of $\alpha$ is the union at a point $P$ of a Prym elliptic curve $(E,\eta)$ and a curve $C_{g-1}\in\cM_{g-1}$ (with trivial 2-torsion bundle) having a 1-dimensional theta-characteristic.
Let us denote by $R_1,R_2,R_3$ the points of $E$ differing from $P$ by $2$-torsion, so that $\eta=\cO_E(R_1-R_2)$. 

Then there are exactly two \textit{limit semicanonical pencils on $E\cup_PC_{g-1}$ remaining even when twisted by $(\eta,\cO_{C_{g-1}})$}. For these limit semicanonical pencils, $\cO_E(R_1-R_3)$ and $\cO_E(R_2-R_3)$ are the induced theta-characteristics on $E$ (and hence $|R_2+P|+(g-3)P$ and $|R_1+P|+(g-3)P$ are the corresponding aspects on $E$).

It follows that the intersection of $\cT^e_g$ and $\Delta_1$ is irreducible (by irreducibility of $\cT_{g-1}$) but not reduced. We also deduce that $\cT^e_g$ will have at most two irreducible components, but we cannot directly derive the irreducibility of $\cT^e_g$.

To circumvent this problem, we consider (as in the proof of \autoref{irredTog}) a Prym curve $(C,\eta)\in\cT^e_g\cap\Delta_1$ obtained by taking $C_{g-1}=C'$ ($C'$ general smooth hyperelliptic curve) and $P\in C'$ a Weierstrass point. Recall that $(C,\eta)$ is the general point of the intersection $\widetilde{\cH}_g\cap\Delta_1$ ($\widetilde{\cH}_g\subset\cT_{g}^e$ being the locus of hyperelliptic Prym curves whose 2-torsion bundle is a difference of two Weierstrass points).

By using monodromy on smooth hyperelliptic curves to interchange the (limit) Weierstrass points $R_1$ and $R_2$, we obtain that monodromy on $\widetilde{\cH}_g\subset\cT_{g}^e$ connects (locally around $(C,\eta)$) the two possible irreducible components of $\cT^e_g$. This finishes the proof.
\end{proof}

\vspace{1mm}

All in all, we have showed the irreducibility of $\cT_g^o$ and $\cT_g^e$ for every $g\neq4$. As explained in the introduction, the irreducibility of $\cT_4^o$ and $\cT_4^e$ can be deduced from a study of the Prym map $\cP_4$ restricted to these divisors, which is contained in \cite{lnr}.

\end{document}